\documentclass[11pt]{article}
\usepackage{amssymb}
\usepackage{}
\usepackage{pifont}

\usepackage{amsfonts}
\usepackage{amsmath, amsthm}
\usepackage{amssymb, amsxtra}

\usepackage{graphicx}

\usepackage{amscd}
\usepackage{multicol}
\usepackage{epstopdf}

\usepackage{epsf}
\usepackage{float}
\usepackage{extarrows}
\usepackage[a4paper, text={.8\paperwidth,.83\paperheight}, ratio=1:1]{geometry}
\usepackage[format=hang,font=small,textfont=it]{caption}
\usepackage{authblk}

\usepackage[colorlinks=true, linkcolor=blue, citecolor=blue]{hyperref}
\usepackage[english]{babel}
\usepackage{latexsym}

\usepackage{tikz}
\usetikzlibrary{calc}
\usepackage[tight]{subfigure}

\usepackage{color}
\definecolor{greenbean}{RGB}{199,237,204}

\newtheorem{thm}{Theorem}[section]

\newtheorem{prop}[thm]{Proposition}

\newtheorem{eg}{Example}[section]

\newtheorem{Def}[thm]{Definition}

\def\C{{\mathbb C}}

\def\P{{\mathbb P}}

\def \ta{\tau}

\def \ta1{\tau_1}

\def \G{\Gamma}

\newcommand{\uThreePointOuter}[5]{
Vertex $ #1 $ is an outer $3$-point. The braid monodromy corresponding to this $3$-point is:
\begin{eqnarray} \label{#5-1}
{\widetilde{\Delta_#1}} = Z_{#2 \; #2', #4 \; #4'}^2 \cdot Z_{#2',#3\;#3'}^3\cdot(Z_{#2\;#2'})^{Z_{#2',#3\;#3'}^2}\cdot(Z_{#3\;#3',#4}^3)^{Z_{#2',#3\;#3'}^2}
\cdot(Z_{#4\;#4'})^{Z_{#3\;#3',#4}^2Z_{#2',#3\;#3'}^2}. \nonumber
\end{eqnarray}
${\widetilde{\Delta_#1}}$ thus gives rise to the following relations:
\begin{equation}\label{#5-2}
\langle\Gamma_#2',\Gamma_#3\rangle=\langle\Gamma_#2',\Gamma_#3'\rangle=\langle\Gamma_#2',
\Gamma_#3^{-1}\Gamma_#3'\Gamma_#3\rangle=e,
\end{equation}
\begin{equation}\label{#5-3}
\Gamma_#2=\Gamma_#3'\Gamma_#3\Gamma_#2'\Gamma_#3^{-1}{\Gamma_#3'}^{-1},
\end{equation}
\begin{equation}\label{#5-4}
\begin{split}
\langle\Gamma_#4,\Gamma_#3'\Gamma_#3\Gamma_#2'\Gamma_#3{\Gamma_#2'}^{-1}\Gamma_#3^{-1}{\Gamma_#3'}^{-1}\rangle=&
\langle\Gamma_#4,\Gamma_#3'\Gamma_#3\Gamma_#2'\Gamma_#3'{\Gamma_#2'}^{-1}\Gamma_#3^{-1}{\Gamma_#3'}^{-1}
\rangle= \\
=& \langle\Gamma_#4,\Gamma_#3'\Gamma_#3\Gamma_#2'\Gamma_#3^{-1}\Gamma_#3'\Gamma_#3{\Gamma_#2'}^{-1}
\Gamma_#3^{-1}{\Gamma_#3'}^{-1}\rangle=e,
\end{split}
\end{equation}
\begin{equation}\label{#5-5}
\Gamma_#4'=\Gamma_#4\Gamma_#3'\Gamma_#3\Gamma_#2'\Gamma_#3'\Gamma_#3{\Gamma_#2'}^{-1}\Gamma_#3^{-1}{\Gamma_#3'}^{-1}\Gamma_#4
\Gamma_#3'\Gamma_#3\Gamma_#2'\Gamma_#3^{-1}{\Gamma_#3'}^{-1}{\Gamma_#2'}^{-1}\Gamma_#3^{-1}
{\Gamma_#3'}^{-1}\Gamma_#4^{-1},
\end{equation}
\begin{equation}\label{#5-6}
[\Gamma_#2,\Gamma_#4]=[\Gamma_#2,\Gamma_#4']=[\Gamma_#2',\Gamma_#4]=[\Gamma_#2',\Gamma_#4']=e.
\end{equation}
}

\newcommand{\uThreePointInner}[5]{
Vertex $ #1 $ is an inner $3$-point. Its braid monodromy is
\begin{eqnarray}
\begin{split} \label{#5-1}
{\widetilde{\Delta_#1}} = & Z_{#2',#3\;#3'}^3 \cdot Z_{#2,#3\;#3'}^3 \cdot (Z_{#3\;#4'})^{Z_{#3\;#3'}^2 Z_{#2',#3\;#3'}^2} \cdot (Z_{#3'\;#4})^{Z_{#3\;#3'}^2 Z_{#4\;#4'}^2 Z_{#2',#3\;#3'}^2} \cdot \\
& (Z_{#3\;#4'})^{Z_{#3\;#3'}^2 Z_{#2,#3\;#3'}^2} \cdot (Z_{#3'\;#4})^{Z_{#3\;#3'}^2 Z_{#4\;#4'}^2 Z_{#2,#3\;#3'}^2}. \nonumber
\end{split}
\end{eqnarray}

These braids give rise to the following relations in $G$:
\begin{equation}\label{#5-1}
\langle\Gamma_#2',\Gamma_#3\rangle=\langle\Gamma_#2',\Gamma_#3'\rangle=\langle\Gamma_#2',
\Gamma_#3^{-1}\Gamma_#3'\Gamma_#3\rangle=e,
\end{equation}
\begin{equation}\label{#5-2}
\langle\Gamma_#2,\Gamma_#3\rangle=\langle\Gamma_#2,\Gamma_#3'\rangle=\langle\Gamma_#2,
\Gamma_#3^{-1}\Gamma_#3'\Gamma_#3\rangle=e,
\end{equation}
\begin{equation}\label{#5-3}
\Gamma_#3'\Gamma_#3\Gamma_#2'\Gamma_#3\Gamma_#2'^{-1}\Gamma_#3^{-1}{\Gamma_#3'}^{-1} = \Gamma_#4',
\end{equation}
\begin{equation}\label{#5-4}
\Gamma_#3'\Gamma_#3\Gamma_#2'\Gamma_#3'\Gamma_#2'^{-1}\Gamma_#3^{-1}{\Gamma_#3'}^{-1} = \Gamma_#4'\Gamma_#4{\Gamma_#4'}^{-1},
\end{equation}
\begin{equation}\label{#5-5}
\Gamma_#3'\Gamma_#3\Gamma_#2\Gamma_#3\Gamma_#2^{-1}\Gamma_#3^{-1}{\Gamma_#3'}^{-1} = \Gamma_#4',
\end{equation}
\begin{equation}\label{#5-6}
\Gamma_#3'\Gamma_#3\Gamma_#2\Gamma_#3'\Gamma_#2^{-1}\Gamma_#3^{-1}{\Gamma_#3'}^{-1} = \Gamma_#4'\Gamma_#4\Gamma_#4'^{-1}.
\end{equation}
}

\newcommand{\uTwoPoint}[4]{
Vertex $ #1 $ is a $2$-point. The braid monodromy corresponding to this point is:
\begin{equation}\label{#4-1}
\widetilde{\Delta_#1} = (Z_{#3\; #3'})^{Z_{#2 \; #2', #3}^2}\cdot Z_{#2 \; #2',#3}^3. \nonumber
\end{equation}
$\widetilde{\Delta_#1}$ gives rise to the following relations:
\begin{equation}\label{#4-2}
\langle \Gamma_#2, \Gamma_#3 \rangle = \langle \Gamma_#2', \Gamma_#3 \rangle = \langle \Gamma_#2^{-1}\Gamma_#2'\Gamma_#2, \Gamma_#3 \rangle = e,
\end{equation}
\begin{equation}\label{#4-3}
\Gamma_#3' = \Gamma_#3\Gamma_#2'\Gamma_#2\Gamma_#3\Gamma_#2^{-1}\Gamma_#2'^{-1}\Gamma_#3^{-1}.
\end{equation}
}

\begin{document}
\title{Fundamental group of Galois covers of degree $5$ surfaces\footnotetext{\hspace{-2.2em}
Email address: Meirav Amram: meiravt@sce.ac.il; Cheng Gong (corresponding author): cgong@suda.edu.cn; \\ Mina Teicher: teicher@macs.biu.ac.il; Wan-Yuan Xu: wanyuanxu@fudan.edu.cn\\
2010 Mathematics Subject Classification. 14D05, 14D06, 14H30, 14J10, 20F36. \\{\bf Key words}: Degeneration, generic projection, Galois cover, braid monodromy, fundamental group
}}
\author[1]{Meirav Amram}
\author[2]{Cheng Gong}
\author[3]{Mina Teicher}
\author[4]{Wan-Yuan Xu}

\affil[1]{\small{Department of Mathematics, Shamoon College of Engineering, Ashdod, Israel}}
\affil[2]{\small{Department of Mathematics, Soochow University, Suzhou 215006, Jiangsu, P. R. of China}}
\affil[3]{\small{Emmy Noether Research Institute for Mathematics, Bar-Ilan University, Ramat-Gan 52900, Israel}}
\affil[4]{\small{ Department of Mathematics, Shanghai Normal University, Shanghai 200234, P. R. of China}}

\date{}

\maketitle

\abstract{
Let $X$ be an algebraic surface of degree $5$, which is considered as a branch cover of $\mathbb{CP}^2$ with respect to a generic projection. The surface has a natural Galois cover with Galois group $S_5$. In this paper, we deal with the fundamental groups of Galois covers of degree $5$ surfaces that degenerate to nice plane arrangements; each of them is a union of five planes such that no three planes meet in a line.
}

\section{Introduction}\label{outline}
In 1977, Gieseker \cite{Gie} proved that the moduli space of surfaces of general type is a quasi-projective variety. Unlike the case of curves, it is not irreducible for fixed numerical invariants.
Catanese and Manetti \cite{C1,C2,Ma} proved some results about the structure and number of components of moduli spaces of general type surfaces. However, even today not much is known about such moduli spaces.
 In \cite{Tei2}, Teicher defined some new invariants of surfaces, stable on connected components of moduli space. The new invariants come from the polycyclic structure of the fundamental group of the complement of a branch curve.

Let $X$ be an algebraic surface of degree $n$; one can consider it as a branched cover of the projective plane $\mathbb{CP}^2$ with respect to a generic projection $\pi : X\rightarrow \mathbb{CP}^2$. The branch curve $S$ is an irreducible cuspidal plane curve of even degree, namely,
$S$ admits only nodes and ordinary cusps as its singularities. Chisini's conjecture \cite{Chisini}, which was confirmed by Kulikov \cite{Ku99, Ku08}, states that: If $S$ is the branch curve of a generic projection $\pi : X\rightarrow \mathbb{CP}^2$, then $\pi$ is determined uniquely by $S$, except for the case when $X$ is the Veronese surface $V_2$ in $\mathbb{CP}^5$. Thus, one can reduce the classification of algebraic surfaces to that of cuspidal branch curves with the same topological type.

It is well-known that the fundamental group of the complement of the branch curve $\pi_1(\mathbb{CP}^2-S)$ does not change when the complex structure of $X$ changes continuously. Thus, we can use such an invariant to distinguish the connected components of the corresponding moduli space of surfaces. In fact, all surfaces in the same component of the moduli space have the same homotopy type and, therefore, the same fundamental group of the complement of the branch curves.

In \cite{MoTe87a} and \cite{MoTe87}, Moishezon-Teicher showed that $\pi_1(\mathbb{CP}^2-S)$ is related with $\pi_1(X_{\text{Gal}})$, where the surface $X_{\text{Gal}}$ is the Galois cover of $X$. Thus, we can calculate fundamental groups of some surfaces of general type that otherwise may be very difficult to determine. Based on this idea, Moishezon-Teicher \cite{MoTe87} constructed a series of simply connected algebraic surfaces of general type, with positive and zero indices, thereby disproving the Bogomolov Watershed conjecture, which held that an algebraic surface of general type, with a positive index, has an infinite fundamental group.

In recent years, much work has been done in the study of $\pi_1(\mathbb{CP}^2-S)$ and $\pi_1(X_{\text{Gal}})$ -- for Cayley's singularities \cite{ADFT}; for different embeddings of $\mathbb{CP}^1\times \mathbb{CP}^1$ \cite{MoTe87}; for the Veronese surfaces $V_n, ~n\geq3$ \cite{19,20}, and $V_2$ \cite{Za37}; for the Hirzebruch surfaces $F_1(a,b),~ F_2(2,2)$ \cite{AFT09,FT08a}; for $T\times T$ where $T$ is a complex torus \cite{Osaka,ATV08}; for $K3$ surfaces \cite{pillow}; for $\mathbb{CP}^1\times T$ \cite{AFT03,AGTV02,AG04}; for $\mathbb{CP}^1\times C_g$ where $C_g$ is a curve of genus $g$ \cite{FT12}; and for certain toric surfaces \cite{Ogata}. In \cite{ADKY}, one can also find a description of computations of braid monodromy and certain quotients of $\pi_1(\mathbb{CP}^2-S)$; the motivation came from the theory of symplectic 4-manifolds. In \cite{Li08}, Liedtke computed a quotient of $\pi_1(X_{\text{Gal}})$ that depends on $\pi_1(X)$ and data from the generic projection only, thereby simplifying the computations of Moishezon, Teicher, and others.
In \cite{A-R-T}, the authors computed the fundamental groups of Galois covers of surfaces of degree $\leq 4$.

This paper may be considered as a  subsequent work of \cite{FT08}, where Friedman and the third named author discussed the local braid monodomry of  a singular point with multiplicity 5 . In this note, we will use the results therein to compute the fundamental groups of the complements of the branch curves and of the Galois covers of degree $5$ surfaces with nice planar degenerations (see Theorems \ref{thm-1}, \ref{thm-2}, \ref{thm-3}, \ref{thm-5}, \ref{thm-6}, \ref{thm-7}, \ref{thm-8}, \ref{thm-4}).

This paper is organized as follows:  In Section \ref{method}, we explain the main methods and give the fundamental and necessary background  that we use in this paper. Section \ref{small} identifies the fundamental groups of the Galois covers of all the surfaces of degree $5$ that degenerate to `nice' planar arrangements, i.e., those in which no three planes meet in a line. We consider the degenerations, give the braid monodromy and the group $\pi_1(\mathbb{CP}^2-S)$, and then determine whether $\pi_1(X_{\text{Gal}})$ is trivial or not.
The surfaces that we consider are: the Hirzebruch surface $F_1(2,1)$ (Subsection \ref{sec1}), a union of the surface $\C\P^1 \times \C\P^1$ and a plane (Subsection \ref{cp1}), a union of the Veronese surface $V_2$ and a plane (Subsection \ref{sec2}), two cases of a union  of the Cayley surface and two planes (Subsection \ref{sec4}), a union of the quartic 4-point surface with a plane (Subsection \ref{extended-4}), the quintic 5-point surface (Subsection \ref{central}), and a 4-point quintic degeneration (Subsection \ref{sec7}).

\paragraph{Acknowledgements:} We thank Ciro Ciliberto for useful
discussions on degeneration of surfaces and for giving us the examples of the surfaces. We also thank Robert Shwartz for helpful comments.

This work is supported by the Emmy Noether Research
Institute for Mathematics, the Minerva Foundation (Germany), the NSFC and the China Postdoctoral Science Foundation.


\section{Method and scientific background}\label{method}
In this section, we describe the main methods and the fundamental background used in this paper. The computations of the fundamental groups $\pi_1(\mathbb{CP}^2-S)$ and $\pi_1(X_{\text{Gal}})$ of surfaces $X$ of degree $5$,  with at worst isolated singularities and with nice degenerations, are explained here.

\subsection{Degeneration and regeneration}
We start with a projective algebraic surface $X$ embedded in a projective space $\mathbb{CP}^n$. We project it onto the projective plane $\mathbb{CP}^2$ to get its branch curve $S$. Because it is not easy to describe $S$, we use a method called degeneration. The definition is as follows:

\begin{Def}
Let $\Delta$ be the unit disc,
and let $X, Y$ be projective algebraic surfaces.
Let $p: Y \rightarrow \mathbb{CP}^2$ and $p': X \rightarrow \mathbb{CP}^2$
be projective embeddings.
We say that $p'$ is a \emph{projective degeneration} of $p$
if there exists a flat family $\pi: V \rightarrow \Delta$
and an embedding $F:V\rightarrow \Delta \times \mathbb{CP}^2$,
such that $F$ composed with the first projection is $\pi$,
and:
\begin{itemize}
\item[(a)] $\pi^{-1}(0) \simeq X$;
\item[(b)] there is a $t_0 \neq 0$ in $\Delta$ such that
$\pi^{-1}(t_0) \simeq Y$;
\item[(c)] the family $V-\pi^{-1}(0) \rightarrow \Delta-{0}$
is smooth;
\item[(d)] restricting to $\pi^{-1}(0)$, $F = {0}\times p'$
under the identification of $\pi^{-1}(0)$ with $X$;
\item[(e)] restricting to $\pi^{-1}(t_0)$, $F = {t_0}\times p$
under the identification of $\pi^{-1}(t_0)$ with $Y$.
\end{itemize}
\end{Def}

We perform a sequence of projective degenerations $X: =X_r \leadsto X_{r-1} \leadsto \cdots X_{r-i} \leadsto
X_{r-(i+1)} \leadsto \cdots \leadsto X_0$, and we refer to each step along the way as a partial
degeneration ($r$ is the number of partial degenerations).  The surface $X_0$ is a total
degeneration of $X$ if it is a union of $\mathbb{CP}^2s$.

In \cite[Sec.\,12]{CCFM} there are
examples of surfaces that can projectively degenerate to a union of planes such that no 3 planes meet in a line.
In our paper we deal with degree $5$ surfaces; this means $n=5$. For more technical details see \cite{amte1}.

\begin{eg}\label{ex:equiv}
Consider the surface $\mathbb{CP}^1\times\mathbb{CP}^1$. Take $\ell_1=\mathbb{CP}^1\times pt$
and $\ell_2=pt\times\mathbb{CP}^1$. For $k,j\in\mathbb{N}$, consider the linear combination $k\ell_1+j\ell_2$.
We embed our surface into a projective space via the linear system $|k\ell_1+j\ell_2|$. Denote the image of the embedding by $X_{kj}$.
Then $X_{kj}$ is a product of 2 rational curves, one of degree $k$ and the other one of degree $j$. Degenerate each of them into $k$  (resp. $j$) lines. Then $X_{kj}$ degenerates to a union of $k\cdot j$ \ $\mathbb{CP}^1\times\mathbb{CP}^1$ (in Figure \ref{cp1cp1}, we have $1\cdot2$ quadrics, each quadric represents  $\mathbb{CP}^1\times\mathbb{CP}^1$).

Note that each quadric can degenerate to a union of 2 planes. It is represented by a diagonal line that divides each square into 2 triangles,  each one representing a plane $\mathbb{CP}^2$ (see Figure \ref{cp1cp1} for the case $k=1,~j=2$).

\begin{figure}[ht]
\begin{center}
\scalebox{0.60}{\includegraphics{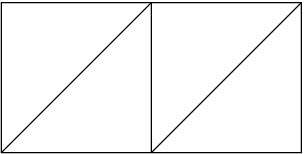}}
\end{center}
\setlength{\abovecaptionskip}{-0.15cm}
\caption{Degeneration of $\C\P^1 \times \C\P^1$}\label{cp1cp1}
\end{figure}

\end{eg}

Consider generic projections $\pi^{(i)} : X_{i} \rightarrow \mathbb{CP}^2$ with the branch curve $S_i$
for $0 \leq i \leq r$. Note that $S_{i-1}$ is a degeneration of
$S_{i}$. We regenerate $S_0$ to get the regenerated branch curve $S:=S_{r}$, using the regeneration Lemmas in \cite{BGT2}.

In the following diagram, 
 we illustrate the connections between the
significant objects $X, X_0, S$, and $S_0$.


\[\begin{CD}
  X\subseteq \mathbb{CP}^n  @>\text{degeneration}>> X_0\subseteq \mathbb{CP}^N \\
  @V\text{generic~ projection}VV                      @VV\text{generic~ projection}V \\
  S\subseteq \mathbb{CP}^2 @<\phantom{regeneration}<\text{regeneration}< S_0\subseteq \mathbb{CP}^2
  \end{CD}\]

\vspace{0.2cm}
 Now we explain in general the regeneration process.
 Say that the degree of the degenerated branch curve $S_{0}$ is $m$.
Each of the $m$ lines of $S_{0}$ should be counted as a double line
in the scheme-theoretic branch locus, as each arises from a nodal line. Another way to see this is to note that
the regeneration of $X_0$ induces a regeneration of $S_{0}$ in such a way that each point on the
typical fiber, say $c$, is replaced by two nearby points $c, c'$. This means that a line regenerates to two parallel lines or to a conic \cite{amte1}. The resulting branch curve $S$ is of degree $2m$.
In full generality, the curve $S_{0}$ may have singularities of multiplicity $k$ for any value of $k$. For simplicity, we say a singularity of multiplicity k is a k-point. A 1-point regenerates to a conic $(j,j')$ with a branch point. A 2-point regenerates to a conic and a tangent line \cite[Claim, p. 8]{19}.  The 3-points and 4-points in the paper can be inner or outer (on the outer edges).  An outer 3-point regenerates to 2 conics intersecting one another and a tangent line to both conics. An inner 3-point regenerates to a double hyperbola and a tangent line. In both cases, each tangency point regenerates to 3 cusps. Outer and inner 4-points are extensions of outer and inner 3-points respectively.
A 5-point regenerates as the 4-point regenerates, but with an addition of the 5th line \cite[Figure 4]{FT08}.

The resulting curve $S$ is a cuspidal curve with nodes and branch points. A branch point is topologically locally equivalent to $y^2=x$; a node (resp. a cusp) is topologically locally equivalent to $y^2=x^2$ (resp. $y^2=x^3$).

We note that 1- and 2-points were considered in \cite{pillow,Ogata,MRT,MoTe87,18,19}, inner 4-points were considered in \cite{pillow,19}, and 5-points were considered in \cite{FT08}.
The regeneration process for a large $k$ can be quite difficult, but work has been done for some specific values: see \cite{FT08}, \cite{amte1}, and \cite{agst} for 5-, 6-, and 8-points, respectively.

Now we need to explain how to derive the related braid monodromy for $S$ and  the fundamental group of its complement in $\mathbb{CP}^2$.
 We will follow the braid monodromy algorithm of Moishezon-Teicher \cite{BGT1,BGT2}. A detailed treatment can also be found in \cite{pillow,amte1}. Note that the braid group (and the braid monodromy) is very useful for study of the projective plane \cite{GG04}.
In the following subsection we present the general setting of a branch curve and the braid monodromy, then we apply it to our case, and to  certain  notations.

\subsection{A general setting of a branch curve}
Consider the setting illustrated in Figure \ref{setup}. Let $S$ be an algebraic curve in $\mathbb{C}^2$, with $p =
\deg(S)$. Let $\pi: \mathbb{C}^2 \rightarrow \mathbb{C}$ be a generic projection onto the first coordinate.
Define the fiber $K(x) = \{y \mid (x,y) \in S\}$ in $S$ over a fixed point $x$, projected to the $y$-axis.
Define $N = \{x \mid \# K(x) < p \}$ and $M' = \{ s \in S \mid \pi_{\mid s} \mbox{ is not \'{e}tale at } s \}$;
note that $\pi (M') = N$. Let $\{A_j\}^q_{j=1}$ be the set of points of $M'$ (the singularities of $S$) and let $N = \{x_j\}^q_{j=1}$ be the
projection of $\{A_j\}^q_{j=1}$ on the $x$-axis. Recall that $\pi$ is generic, so we assume that $\# (\pi^{-1}(x) \cap M') =1$ for
every $x \in N$. Let $E$ (resp. $D$) be a closed disk on the $x$-axis (resp. the $y$-axis), such that $M'
\subset \mbox{Int}(E \times D)$ and $N \subset \mbox{Int}(E)$. We choose $u \in \partial E$, a real point ``far enough" from the
set $N$,
 so that $|x|\ll u$ for every $x \in N$.
Define $\mathbb{C}_u = \pi^{-1}(u)$ and number the points of
$K=\mathbb{C}_u\cap S$ as $\{1 , \dots , p\}$.

\begin{figure}[ht]
\begin{center}
\scalebox{1.1}{\includegraphics{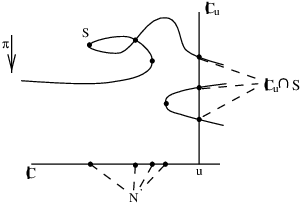}}
\end{center}
\setlength{\abovecaptionskip}{-0.15cm}
\caption{General setting}\label{setup}
\end{figure}

We now construct a geometric-base (g-base) for the fundamental group $\pi_1(E - N, u)$ (see \cite{BGT1} for details). Take a set of paths
$\{\gamma_j\}^q_{j=1}$ that connect $u$ with the points $\{x_j\}^q_{j=1}$ of $N$. Now encircle each $x_j$ with
a small circle, $c_j$, oriented counterclockwise. Denote the path segment from $u$ to the boundary of this circle
by $\gamma'_j$. We define an element (a loop) in the g-base as $\delta_j = {\gamma'}_j c_j {\gamma'}^{-1}_j$.
Let $B_p[D,K]$ be the braid group, and let $H_1, \dots , H_{p-1}$ be its frame (for complete definitions, see
\cite[Section III.2]{BGT1}). The braid monodromy of $S$  is a map $\varphi: \pi_1(E - N, u)
\rightarrow B_p[D,K]$ defined as follows (see \cite{artin} in detail): every loop in $E - N$ starting at $u$ has liftings to a system of $p$ paths in $(E - N) \times D$ starting at each point of $K$. Projecting them to $D$, we obtain $p$
paths in $D$ defining a motion $\{1(t), \dots , p(t)\}$ (for $0 \leq t \leq 1$) of $p$ points in $D$ starting
and ending at $K$. This motion defines a braid in $B_p [D , K]$.

We now explain how to compute this braid, following the notation of
Moishezon-Teicher. Let $A_j$ be a singularity in $S$ and $x_j$ its projection by $\pi$ to the $x$-axis. We
choose a point $x'_j$ next to $x_j$, such that \ $\pi^{-1}(x'_j)$ is a typical fiber.  We encircle $A_j$ with a
very small circle in such a way that the typical fiber $\pi^{-1}(x'_j)$ intersects the circle at two points, say
$a ,b$. We fix a skeleton (see \cite{BGT2}) $\xi_{x'_j}$ that connects $a$ and $b$, and denote it as $\langle a,b \rangle$. The
Lefschetz diffeomorphism $\Psi$ (see \cite{BGT1}) defines the corresponding skeleton $(\xi_{x'_j}) \Psi$ in the
typical fiber $\mathbb{C}_u$. This one defines a motion of its two endpoints, and we get the braid $\varphi(\delta_j) = \Delta \langle (\xi_{x'_j}) \Psi \rangle^{\epsilon_j}$, where $\varepsilon$ is fixed according to the type of the singularity (locally defined by the equation $y^2=x^\epsilon$, $\epsilon=1,2$, or $3$).
The braid monodromy factorization associated to $S$ is $\Delta^2_{p} = \prod\limits^q_{j=1} \varphi(\delta_j)$.

Using the braid monodromy factorization, we compute the fundamental group of the complement of $S$. By the van
Kampen Theorem \cite{vk}, there is a ''good" geometric base (g-base) $\{\G_{j}\}$ of $\pi_{1}(\C_{u} - S \cap \C_{u}, *)$,
such that the fundamental group $\pi_{1}(\mathbb{CP}^2 - S)$  of the complement of $S$ in $\mathbb{CP}^2$  is generated by the
images of $\{ \G_{j} \}$ with the relations $\varphi(\delta_{i})\G_{j} = \G_{j} \hspace{0.2cm}
\forall\, i,j$.

For our purposes, we take the curve $S$ to be the branch curve of a smooth surface $X$, which is a
nodal-cuspidal curve. Consider a small circle around a singularity. Denote by $a$ and $b$
the intersection points of the two branches with this small circle. Note that the branches meet at the singularity. Let $\G_{a},
\G_{b}$ be two non-intersecting loops in $\pi_{1}(\C_{u} - S \cap \C_{u}, *)$ around the intersection points of
the branches with the fiber $\C_{u}$ (constructed by cutting each of the paths and creating two loops that
proceed along the two parts and encircle $a$ and $b$); see \cite[Proposition-Example VI.1.1]{BGT1}.

\subsection{Braids and fundamental groups in our paper}
We return to the data described in this paper. Using the above explanations, in the regeneration process a 1-point regenerates to a conic $(j,j')$, a node between two lines $i$ and $j$ becomes two nodes as line $j$ doubles into two parallel lines, and a tangency between a conic $(i,i')$ and  line $j$ regenerates to three cusps when $j$ doubles into two parallel lines. Recall now that $S$ is a cuspidal curve of degree $2m$ (after the regeneration process).

 Therefore the braids related to $S$ in the paper are denoted as follows (we use the same notations as in \cite{19}):
\begin{enumerate}
\item for a branch point, $Z_{j \;j'}$ is a counterclockwise
half-twist of $j$ and $j'$ along a path below the real
axis,
\item for nodes, $Z^2_{i, j \;j'}=Z_{i\;j}^2 \cdot Z_{i\;j'}^2$ and
$Z^2_{i\;i', j \;j'}=Z_{i\;j}^2 \cdot Z_{i'\;j}^2\cdot Z_{i\;j'}^2\cdot Z_{i'\;j'}^2$,
\item for cusps, $Z^3_{i, j \;j'}=Z^3_{i\; j} \cdot (Z^3_{i\; j})^{Z_{j\;j'}} \cdot (Z^3_{i \; j})^{Z^{-1}_{j \; j'}}$.
\end{enumerate}
We note that a conjugation of braids is defined as $a^b = b^{-1}ab$.

Moreover, we denote the generators of the group $\pi_1(\mathbb{CP}^2-S)$ as $\G_1, \G_{1}', \dots, \G_{2m}, \G_{2m}'$. By the van Kampen
Theorem \cite{vk} we can get a presentation of $\pi_1(\mathbb{CP}^2-S)$ by means of generators $\{\G_{j}, \G_{j}'\}$ and relations of the types:
\begin{enumerate}
\item for a branch point, $\G_{j} = \G_{j}'$,
\item for nodes,
\begin{itemize}
\item $[\G_{i},\G_{j}]=\G_{i}\G_{j}\G_{i}^{-1}\G_{j}^{-1}=e$,
\item $[\G_{i},\G_{j}']=\G_{i}\G_{j}'\G_{i}^{-1}\G_{j}'^{-1}=e$,
\end{itemize}
\item for cusps,
\begin{itemize}
\item $\langle\G_{i},\G_{j}\rangle=\G_{i}\G_{j}\G_{i}\G_{j}^{-1}\G_{i}^{-1}\G_{j}^{-1}=e$,
\item $\langle\G_{i},\G_{j}'\rangle=\G_{i}\G_{j}'\G_{i}\G_{j}'^{-1}\G_{i}^{-1}\G_{j}'^{-1}=e$,
\item $\langle\G_{i},\G_{j}^{-1}\G_{j}'\G_{j}\rangle=\G_{i}\G_{j}^{-1}\G_{j}'\G_{j}\G_{i}\G_{j}^{-1}\G_{j}'^{-1}
    \G_{j}\G_{i}^{-1}\G_{j}^{-1}\G_{j}'^{-1}\G_{j}=e$.
    \end{itemize}
\end{enumerate}

To each list of relations we add the projective relation $\prod\limits_{j=m}^1 \G_{j}'\G_{j}=e$. Moreover, in some cases in the paper, we have parasitic intersections that induce commutative relations. These intersections come from lines in $X_0$ that do not intersect, but when projecting $X_0$ onto $\C\P^2$, they will intersect. See details in \cite{MoTe87}.

\bigskip

Our techniques also allow us to compute fundamental groups of Galois covers. We recall from \cite{MoTe87} that if
 $f : X\rightarrow \mathbb{CP}^2$ is a generic projection of degree $n$, then ${X}_{\text{Gal}}$, the Galois cover, is defined as follows:
  $${X}_{\text{Gal}}=\overline{(X \times_{\mathbb{CP}^{2}} \ldots \times_{\mathbb{CP}^{2}} X)-\triangle},$$
where the product is taken $n$ times, and $\triangle$ is the diagonal. We denote $G=\pi_1(\mathbb{CP}^2-S)$, and consider the quotient group $\tilde{G} = G/{\langle \G_{j}^2=e,\G_{j}'^2=e \rangle}$. We apply the theorem of Moishezon-Teicher. There is an exact sequence
\begin{equation}\label{M-T}
0 \rightarrow  \pi_1(X_{\text{Gal}}) \rightarrow \tilde{G} \rightarrow S_n \rightarrow 0,
\end{equation}
where the fundamental group of the Galois cover $\pi_1(X_{\text{Gal}})$ is the kernel of $\tilde{G}$ projected onto $S_n$.
In this paper, we determine whether $\pi_1(X_{\text{Gal}})$ is trivial or not.

\section{Calculations of the fundamental group}\label{small}
In this section, we consider the Hirzebruch surface $F_{1}(2,1)$ (Subsection \ref{sec1}), a union of the surface $\C\P^1 \times \C\P^1$ and a plane (Subsection \ref{cp1}), a union of the Veronese surface $V_2$ and a plane (Subsection \ref{sec2}), two cases of a union of the Cayley surface and two planes (Subsection \ref{sec4}), a union of the quartic 4-point surface with a plane (Subsection \ref{extended-4}), the quintic 5-point surface (Subsection \ref{central}), and the 4-point quintic degeneration (Subsection \ref{sec7}).

\begin{prop} There are $8$ possible quintic degenerations such that no 3 planes meet at a line, corresponding to Figures \ref{case1}, \ref{case3}, \ref{case2}, \ref{case4}, \ref{case7}, \ref{case6}, \ref{case5}, and \ref{case8}.
\end{prop}

\begin{proof} We construct the degenerations combinatorially by gluing 5 triangles such that no 3 triangles meet in a line (each triangle representing a projective plane). It is not hard to see that there are eight possible cases. To be precise, if there is a 5-point, then the configuration of the degeneration is Figure \ref{case5}. If there is a 4-point, then there are 2 cases, one is the degeneration of the 4-point surface with a plane (Figure \ref{case6}), the other one is the degeneration of a 4-point quintic surface (Figure \ref{case8}). If there is a 3-point, then the configurations of the degenerations are Figures \ref{case3}, \ref{case2}, \ref{case4} and \ref{case7}. Figure \ref{case3} is the degeneration of $\mathbb{CP}^1 \times \mathbb{CP}^1$ with a plane. Figure \ref{case2} is the degeneration of the Veronese surface with a plane. Figures \ref{case4} and \ref{case7} are the degenerations of the Cayley surface and two other planes. If the degeneration contains
only 2-points and 1-points, then the configuration of the degeneration is Figure \ref{case1}.
\end{proof}

\subsection{The Hirzebruch surface $F_1(2,1)$}\label{sec1}
Let $F_1$ be the 1st Hirzebruch surface, i.e., the projection of the vector bundle $\mathcal {O}_{\mathbb{CP}^1}(1)\bigoplus \mathcal {O}_{\mathbb{CP}^1}$.
Denote by $s$ the holomorphic section of $\mathcal {O}_{\mathbb{CP}^1}(1)$, and by $E_0\subset F_1$ the image of the section $(s,1)$ of $\mathcal {O}_{\mathbb{CP}^1}(1)\bigoplus \mathcal {O}_{\mathbb{CP}^1}$. The Picard group is always generated by the fiber $C$ and $E_0$. Note that $2C+E_0$ is very ample and thus defines an embedding $f_{\mid2C+E_0\mid}: F_1 \to \mathbb{CP}^N$. Let $F_1(2,1)=f_{\mid2C+E_0\mid}(F_1)$. By the constructions in \cite{MRT},
$F_1(2,1)$ degenerates to a union of five planes, as depicted in Figure \ref{case1}.
\vspace{0.2cm}

\begin{figure}[ht]
\begin{center}
\scalebox{0.60}{\includegraphics{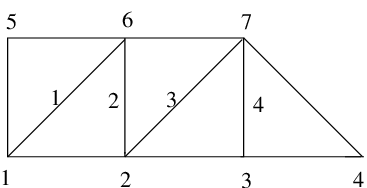}}
\end{center}
\setlength{\abovecaptionskip}{-0.15cm}
\caption{Degeneration of the Hirzebruch surface $F_1(2,1)$}\label{case1}
\end{figure}

\begin{thm}\label{thm-1}
The fundamental group of the Galois cover $\pi_1(X_{\text{Gal}})$ of the Hirzebruch surface $F_1(2,1)$ is trivial.
\end{thm}

\begin{proof}
See \cite[Theorem\,0.1]{MRT}

\end{proof}

\subsection{The union of $\mathbb{CP}^1 \times \mathbb{CP}^1$ degeneration and a plane}\label{cp1}
In this subsection we investigate the surface whose degeneration is depicted in Figure \ref{case3}, i.e., the union of the $\C\P^1 \times \C\P^1$ type degeneration with a plane.
In the degeneration, one can see the common edge, numbered as $1$.

\begin{figure}[H]
\begin{center}
\scalebox{0.60}{\includegraphics{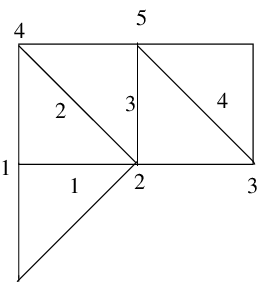}}
\end{center}
\setlength{\abovecaptionskip}{-0.15cm}
\caption{The union of $\C\P^1 \times \C\P^1$ degeneration and a plane}\label{case3}
\end{figure}

\begin{thm}\label{thm-2}
The fundamental group of the Galois cover of the surface with the degeneration as in Figure \ref{case3} is trivial.
\end{thm}

\begin{proof}
The branch curve $S$ in $\mathbb{CP}^2$ is an arrangement of 8 lines. We regenerate each vertex in turn and compute the group $G$.

Vertices 1, 3, and 4 are 1-points; therefore, they give rise to the braids  $Z_{1\; 1^{'}}$, $Z_{4\; 4^{'}}$, and $Z_{2\; 2^{'}}$, respectively, and hence to the following relations in $G$:
\begin{equation}\label{4-1}
\Gamma_1=\Gamma_1',~~~\Gamma_4=\Gamma_4',~~~\Gamma_2=\Gamma_2'.
\end{equation}
Vertex 5  is a 2-point that gives rise to the braid monodromy factors
 $Z^3_{3\; 3^{'},4}$, $(Z_{4\; 4^{'}})^{Z^2_{3\; 3^{'},4}}$ and to the following relations:
\begin{equation}\label{4-2}
\langle\Gamma_3,\Gamma_4\rangle=\langle\Gamma_3',\Gamma_4\rangle=\langle\Gamma_3^{-1}\Gamma_3'\Gamma_3,\Gamma_4\rangle=e,
\end{equation}
\begin{equation}\label{4-3}
\Gamma_4'=\Gamma_4\Gamma_3'\Gamma_3\Gamma_4\Gamma_3^{-1}{\Gamma_3'}^{-1}\Gamma_4^{-1}.
\end{equation}
Vertex 2 is an outer 3-point. The braid monodromy corresponding to this 3-point is:
\begin{eqnarray} \label{4-4}
{\widetilde{\Delta_2}} = Z_{1 \;1', 3\;3'}^2 \cdot Z_{1',2\;2'}^3\cdot(Z_{1\;1'})^{Z_{1',2\;2'}^2}\cdot(Z_{2\;2',3}^3)^{Z_{1',2\;2'}^2}
\cdot(Z_{3\;3'})^{Z_{2\;2',3}^2 Z_{1',2\;2'}^2}. \nonumber
\end{eqnarray}
${\widetilde{\Delta_2}}$ thus gives rise to the following relations:
\begin{equation}\label{4-5}
\langle\Gamma_1',\Gamma_2\rangle=\langle\Gamma_1',\Gamma_2'\rangle=\langle\Gamma_1',\Gamma_2^{-1}\Gamma_2'\Gamma_2\rangle=e,
\end{equation}
\begin{equation}\label{4-6}
\Gamma_1=\Gamma_2'\Gamma_2\Gamma_1'\Gamma_2^{-1}{\Gamma_2'}^{-1},
\end{equation}
\begin{equation}\label{4-7}
\langle\Gamma_3,\Gamma_2'\Gamma_2\Gamma_1'\Gamma_2{\Gamma_1'}^{-1}\Gamma_2^{-1}{\Gamma_2'}^{-1}\rangle=
\langle\Gamma_3,\Gamma_2'\Gamma_2\Gamma_1'\Gamma_2'{\Gamma_1'}^{-1}\Gamma_2^{-1}{\Gamma_2'}^{-1}\rangle=
\langle\Gamma_3,\Gamma_2'\Gamma_2\Gamma_1'\Gamma_2^{-1}\Gamma_2'\Gamma_2{\Gamma_1'}^{-1}\Gamma_2^{-1}{\Gamma_2'}^{-1}\rangle=e,
\end{equation}
\begin{equation}\label{4-8a}
\Gamma_3'=\Gamma_3\Gamma_2'\Gamma_2\Gamma_1'\Gamma_2'\Gamma_2{\Gamma_1'}^{-1}\Gamma_2^{-1}{\Gamma_2'}^{-1}\Gamma_3
\Gamma_2'\Gamma_2\Gamma_1'\Gamma_2^{-1}{\Gamma_2'}^{-1}{\Gamma_1'}^{-1}\Gamma_2^{-1}{\Gamma_2'}^{-1}\Gamma_3^{-1},
\end{equation}
\begin{equation}\label{4-8b}
[\Gamma_1,\Gamma_3]=[\Gamma_1,\Gamma_3']=[\Gamma_1',\Gamma_3]=[\Gamma_1',\Gamma_3']=e.
\end{equation}
We also have the following parasitic and projective relations:
\begin{equation}\label{4-9}
[\Gamma_1,\Gamma_4]=[\Gamma_1,\Gamma_4']=[\Gamma_1',\Gamma_4]=[\Gamma_1',\Gamma_4']=e,
\end{equation}
\begin{equation}\label{4-10}
[\Gamma_2,\Gamma_4]=[\Gamma_2,\Gamma_4']=[\Gamma_2',\Gamma_4]=[\Gamma_2',\Gamma_4']=e,
\end{equation}
\begin{equation}\label{4-11}
\Gamma_4'\Gamma_4\Gamma_3'\Gamma_3\Gamma_2'\Gamma_2\Gamma_1'\Gamma_1=e.
\end{equation}

Relation \eqref{4-3} is simplified by \eqref{4-1} and \eqref{4-2}:
\begin{eqnarray*}
 & \Gamma_4=\Gamma_4\Gamma_3'\Gamma_3\Gamma_4\Gamma_3^{-1}{\Gamma_3'}^{-1}\Gamma_4^{-1} \Rightarrow \Gamma_4=\Gamma_3'\Gamma_3\Gamma_4\Gamma_3^{-1}{\Gamma_3'}^{-1} \Rightarrow \\ & {\Gamma_3'}^{-1}\Gamma_4\Gamma_3'=\Gamma_3\Gamma_4\Gamma_3^{-1} \Rightarrow \Gamma_4\Gamma_3'\Gamma_4^{-1}=\Gamma_4^{-1}\Gamma_3\Gamma_4 \nonumber
\end{eqnarray*}
and we get
\begin{equation}\label{4-12}
\Gamma_3'=\Gamma_4^{-2}\Gamma_3\Gamma_4^2.
\end{equation}
Substituting this expression in the other relations, we get the simplified presentation of $G$:
$$\Gamma_1=\Gamma_1',\ \Gamma_2=\Gamma_2',\ \Gamma_4=\Gamma_4',\  \Gamma_3'=\Gamma_4^{-2}\Gamma_3\Gamma_4^2,$$
$$\langle\Gamma_1,\Gamma_2\rangle=\langle\Gamma_2,\Gamma_3\rangle=\langle\Gamma_3,\Gamma_4\rangle=e,$$
$$[\Gamma_1,\Gamma_3]=[\Gamma_1,\Gamma_4]=[\Gamma_2,\Gamma_4]=e,$$
$$\Gamma_1 = \Gamma_2^2\Gamma_1\Gamma_2^{-2},$$
$$\Gamma_3\Gamma_4^2\Gamma_3\Gamma_2^2\Gamma_1^2=e.$$

Then $\tilde{G}$ admits the following presentation
\begin{equation}
\Gamma_1^2=\Gamma_2^2=\Gamma_3^2=\Gamma_4^2=e,
\end{equation}
$$\langle\Gamma_1,\Gamma_2\rangle=\langle\Gamma_2,\Gamma_3\rangle=\langle\Gamma_3,\Gamma_4\rangle=e,$$
$$[\Gamma_1,\Gamma_3]=[\Gamma_1,\Gamma_4]=[\Gamma_2,\Gamma_4]=e.$$
Since $\tilde{G}=\{\Gamma_1, \Gamma_2,\Gamma_3,\Gamma_4 | ~~S_5~~ \text{type relations}\}\cong S_5$, then $\pi_1(X_{\text{Gal}})$ is trivial.
\end{proof}

\subsection{The union of the Veronese $V_2$ degeneration and a plane}\label{sec2}
We denote by $V_2$ the Veronese surface of order $2$.
In this subsection we introduce the degeneration of a union of the Veronese surface $V_2$ and a plane. The degeneration of this surface is a union of five planes, where $V_2$ and the plane are united along the edge, numbered as 1, see Figure \ref{case2}.

We note that the surface $V_2$ is atypical in different algebraic-geometrical theories, for example, it is the exception case for Chisini's conjecture \cite{Chisini}.

\begin{figure}[H]
\begin{center}
\scalebox{0.50}{\includegraphics{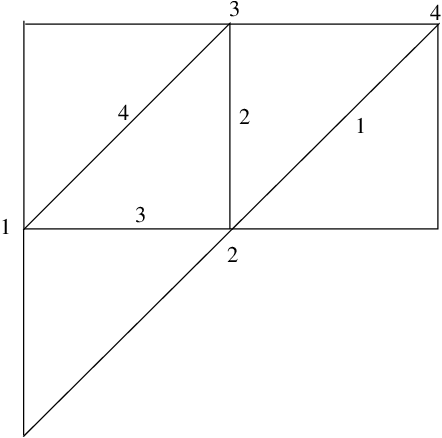}}
\end{center}
\setlength{\abovecaptionskip}{-0.15cm}
\caption{The union of $V_2$ degeneration and a plane}\label{case2}
\end{figure}

\begin{thm}\label{thm-3}
The fundamental group of the Galois cover of the surface with the degeneration as in Figure \ref{case2} is not trivial.
\end{thm}

\begin{proof}
The branch curve $S$ in $\mathbb{CP}^2$ is an arrangement of 8 lines. We regenerate each vertex in turn and compute the group $G$.

Vertex 4 is a 1-point that gives rise to braid  $Z_{1 \; 1^{'}}$, and derives the following relation in $G$:
\begin{equation}\label{3-1}
\Gamma_1=\Gamma_1'.
\end{equation}
Vertex 3 (resp. 1) is a 2-point that gives rise to the braid monodromy factors
$Z^3_{2\; 2^{'},4}$ and $(Z_{4\; 4^{'}})^{Z^2_{2\;2^{'},4}}$ (resp. $Z^3_{3^{'},4\;4^{'}}$ and $(Z_{3\;3^{'}})^{Z^2_{3^{'},4\;4^{'}}}$).
 These braids yield the following relations:
\begin{equation}\label{3-2}
\langle\Gamma_3,\Gamma_4\rangle=\langle\Gamma_3',\Gamma_4\rangle=\langle\Gamma_3^{-1}\Gamma_3'\Gamma_3,
\Gamma_4\rangle=e,
\end{equation}
\begin{equation}\label{3-3}
\Gamma_4'=\Gamma_4\Gamma_3'\Gamma_3\Gamma_4\Gamma_3^{-1}\Gamma_3'^{-1}\Gamma_4^{-1},
\end{equation}
\begin{equation}\label{3-4}
\langle\Gamma_2',\Gamma_4\rangle=\langle\Gamma_2',\Gamma_4'\rangle=\langle\Gamma_2',\Gamma_4^{-1}
\Gamma_4'\Gamma_4\rangle=e,
\end{equation}
\begin{equation}\label{3-5}
\Gamma_2=\Gamma_4'\Gamma_4\Gamma_2'\Gamma_4^{-1}{\Gamma_4'}^{-1}.
\end{equation}
\uThreePointOuter{2}{1}{2}{3}{union-vert2}
We also have the following parasitic and projective relations:
\begin{equation}\label{3-12}
[\Gamma_1,\Gamma_4]=[\Gamma_1,\Gamma_4']=[\Gamma_1',\Gamma_4]=[\Gamma_1',\Gamma_4']=e,
\end{equation}
\begin{equation}\label{3-13}
\Gamma_4'\Gamma_4\Gamma_3'\Gamma_3\Gamma_2'\Gamma_2\Gamma_1'\Gamma_1=e.
\end{equation}

We consider now $\tilde{G}$. Using \eqref{3-1}, \eqref{union-vert2-2}, and  \eqref{union-vert2-3}, we can get $\Gamma_2=\Gamma_2'$. Then by \eqref{union-vert2-5}, we get also $\Gamma_3=\Gamma_3'$.  From \eqref{3-3} we get $\Gamma_4=\Gamma_4'$.

The relations in $\tilde{G}$ are now
\begin{equation}
\Gamma_1^2=\Gamma_2^2=\Gamma_3^2=\Gamma_4^2=e,
\end{equation}
 \begin{equation}
\langle\Gamma_1,\Gamma_2\rangle=\langle\Gamma_2,\Gamma_3\rangle=\langle\Gamma_2,\Gamma_4\rangle=
\langle\Gamma_3,\Gamma_4\rangle=e,
\end{equation}
\begin{equation}
[\Gamma_1,\Gamma_3]=[\Gamma_1,\Gamma_4]=e.
\end{equation}
Since the relation $[\Gamma_2\Gamma_3\Gamma_2,\Gamma_4]=e$ is missing, it means that $\pi_1(X_{\text{Gal}})$  is not trivial.

\end{proof}

\subsection{The union of the Cayley degeneration and two planes}\label{sec4}
The classification of singular cubic surfaces in $\C\P^3$ was done
in the 1860s, by Schl\"{a}fli \cite{Sch} and Cayley \cite{Cay}.
Surface XVI in Cayley's classification is now called the Cayley
cubic, and when embedded in $\C\P^3$, it is defined by the following
equation:
$$
4 (X^3 + Y^3 + Z^3 + W^3) - (X + Y + Z + W)^3=0.
$$
It has four singularities, which are ordinary double points.
Cayley noticed that this surface is a unique cubic surface having four
ordinary double points, which is the maximal possible number of
double points for a cubic surface (see, for example, Salmon's book
\cite{Sal}).

In this subsection we introduce two cases of a degeneration that is a union of the Cayley surface and two planes. We call them Type I (Figure \ref{case4}) and Type II (Figure \ref{case7}).

\subsubsection{Type I}
 In this subsection we consider the union of the Cayley surface and two planes (Type I).
 The degeneration of this surface is a union of five planes, see Figure \ref{case4}.

 The Hilbert scheme of del Pezzo surfaces of degree 5 contains configurations of planes as in Figure \ref{case4} (degenerations of Type I). Indeed, the rational normal scroll $F$ degenerates to the 4 ``top'' planes in Figure \ref{case4} (this is a toric degeneration, corresponding to the subdivision of the $(1,2)$--rectangle). Then the ``bottom'' length 2 polyline $34$ in Figure \ref{case4} is a conic $C$, which degenerates here in two lines corresponding to the segments (edges) $4$ and $5$. The plane $P$ of $C$ becomes the plane spanned by edges $4$ and $5$, which is just the ``bottom plane'' in Figure \ref{case4}.

\begin{figure}[H]
\begin{center}
\scalebox{0.50}{\includegraphics{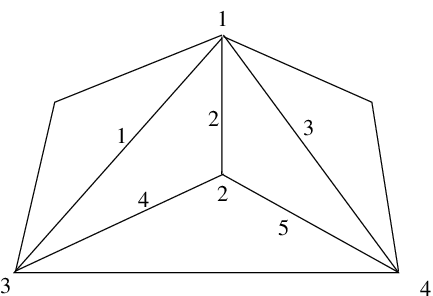}}
\end{center}
\setlength{\abovecaptionskip}{-0.15cm}
\caption{Degeneration of Type I}\label{case4}
\end{figure}

\begin{thm}\label{thm-5}
The fundamental group of the Galois cover of the surface with the degeneration as in Figure \ref{case4} is trivial.
\end{thm}

\begin{proof}
The branch curve $S$ in $\mathbb{CP}^2$ is an arrangement of 10 lines. We regenerate each vertex in turn and compute the group $G$.

\uTwoPoint{3}{1}{4}{C-vert3}
\uTwoPoint{4}{3}{5}{C-vert4}
\uThreePointOuter{1}{1}{2}{3}{C-vert1}
\uThreePointInner{2}{2}{4}{5}{C-vert2}
We also have the following parasitic and projective relations:
\begin{equation}\label{C-par1}
[\Gamma_1,\Gamma_5]=[\Gamma_1,\Gamma_5']=[\Gamma_1',\Gamma_5]=[\Gamma_1',\Gamma_5']=e,
\end{equation}
\begin{equation}\label{C-par2}
[\Gamma_3,\Gamma_4]=[\Gamma_3,\Gamma_4']=[\Gamma_3',\Gamma_4]=[\Gamma_3',\Gamma_4']=e,
\end{equation}
\begin{equation}
\Gamma_5'\Gamma_5\Gamma_4'\Gamma_4\Gamma_3'\Gamma_3\Gamma_2'\Gamma_2\Gamma_1'\Gamma_1=e.
\end{equation}

Equating \eqref{C-vert2-3} and \eqref{C-vert2-5} in $\tilde{G}$, we can get $\Gamma_2=\Gamma_2'$.
Then by \eqref{C-vert1-3}, we get $\Gamma_1=\Gamma_1'$. A direct result from \eqref{C-vert1-5} is that $\Gamma_3=\Gamma_3'$. Looking at the relations from the 2-points, we can deduce also that
$\Gamma_4=\Gamma_4'$ and $\Gamma_5=\Gamma_5'$.

The relations in $\tilde{G}$ are now
\begin{equation}
\Gamma_1^2=\Gamma_2^2=\Gamma_3^2=\Gamma_4^2=\Gamma_5^2=e,
\end{equation}
 \begin{equation}
\langle\Gamma_1,\Gamma_2\rangle=\langle\Gamma_1,\Gamma_4\rangle=\langle\Gamma_2,\Gamma_3\rangle=
\langle\Gamma_2,\Gamma_4\rangle=\langle\Gamma_2,\Gamma_5\rangle=\langle\Gamma_3,\Gamma_5\rangle=
\langle\Gamma_4,\Gamma_5\rangle=e,
\end{equation}
\begin{equation}
[\Gamma_1,\Gamma_3]=[\Gamma_1,\Gamma_5]=[\Gamma_3,\Gamma_4]=e,
\end{equation}
\begin{equation}
\Gamma_5=\Gamma_2\Gamma_4\Gamma_2.
\end{equation}
We can get the following relations too:
\begin{equation}
[\Gamma_1,\Gamma_2\Gamma_4\Gamma_2]=[\Gamma_3,\Gamma_2\Gamma_5\Gamma_2]=e.
\end{equation}
According to these relations, $\tilde{G}$ with the generators $\Gamma_1,\Gamma_2,\Gamma_3,\Gamma_4,\Gamma_5$ is isomorphic to $S_5$, so $\pi_1(X_{\text{Gal}})$  is trivial.
\end{proof}

\subsubsection{Type II}\label{typeII}
In this subsection we consider the union of the Cayley surface and two planes (Type II).
 The degeneration of this surface is a union of five planes, where the common edges of the Cayley degeneration and one of the planes is 4, and two planes have a common edge that is 5, see Figure \ref{case7}.

The Hilbert scheme of del Pezzo surfaces of degree 5 contains configurations of planes as in Figure  \ref{case7} (degenerations of Type II). Indeed, the Veronese  $V$  degenerates to the 4 ''rightmost" planes (except a plane with vertices 123) in Figure
 \ref{case7} (this is a toric degeneration).
Then in vertex 3, the length 2 polyline $13$ of $V$ corresponds to a conic $C$, which degenerates here in two lines corresponding to the segments (edges) $1$ and $2$. The plane $P$ of $C$ becomes the plane spanned by edges $1$ and $2$, which is just a plane with vertices 123 in Figure  \ref{case7}.

\begin{figure}[H]
\begin{center}
\scalebox{0.50}{\includegraphics{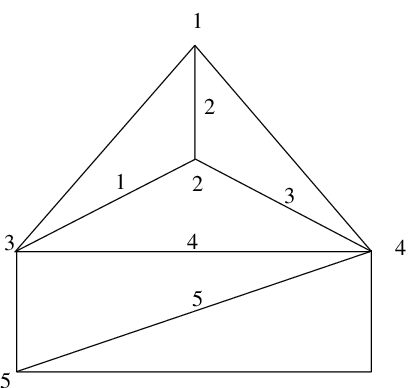}}
\end{center}
\setlength{\abovecaptionskip}{-0.15cm}
\caption{Degeneration of Type II}\label{case7}
\end{figure}

\begin{thm}\label{thm-6}
The fundamental group of the Galois cover of the surface with the degeneration as in Figure \ref{case7} is trivial.
\end{thm}

\begin{proof}

The branch curve $S$ in $\mathbb{CP}^2$ is an arrangement of 10 lines. We regenerate each vertex in turn and compute the group $G$.

Vertices 1 and 5 give rise to the braids  $Z_{2\; 2^{'}}$ and $Z_{5\; 5^{'}}$ respectively, and hence to the relations:
\begin{equation}\label{8-1}
\Gamma_2=\Gamma_2',~~~~~~\Gamma_5=\Gamma_5'.
\end{equation}
\uTwoPoint{3}{1}{4}{II-vert3}
\uThreePointOuter{4}{3}{4}{5}{II-vert4}
\uThreePointInner{2}{1}{2}{3}{II-vert2}
We also have the following parasitic and projective relations:
\begin{equation}
[\Gamma_1,\Gamma_5]=[\Gamma_1',\Gamma_5]=[\Gamma_1,\Gamma_5']=[\Gamma_1',\Gamma_5']=e,
\end{equation}
\begin{equation}
[\Gamma_2,\Gamma_4]=[\Gamma_2',\Gamma_4]=[\Gamma_2,\Gamma_4']=[\Gamma_2',\Gamma_4']=e,
\end{equation}
\begin{equation}
[\Gamma_2,\Gamma_5]=[\Gamma_2',\Gamma_5]=[\Gamma_2,\Gamma_5']=[\Gamma_2',\Gamma_5']=e,
\end{equation}
\begin{equation}
\Gamma_5'\Gamma_5\Gamma_4'\Gamma_4\Gamma_3'\Gamma_3\Gamma_2'\Gamma_2\Gamma_1'\Gamma_1=e.
\end{equation}

Equating \eqref{II-vert2-3} and \eqref{II-vert2-5} in $\tilde{G}$, we can get $\Gamma_1=\Gamma_1'$.
Substituting this in \eqref{II-vert3-3}, we get easily $\Gamma_4=\Gamma_4'$. Then from \eqref{II-vert4-3} we get also $\Gamma_3=\Gamma_3'$.

Thus the relations in $\tilde{G}$ are
\begin{equation}
\Gamma_1^2=\Gamma_2^2=\Gamma_3^2=\Gamma_4^2=\Gamma_5^2=e,
\end{equation}
\begin{equation}
\langle\Gamma_1,\Gamma_2\rangle=\langle\Gamma_1,\Gamma_3\rangle=\langle\Gamma_1,\Gamma_4\rangle=
\langle\Gamma_2,\Gamma_3\rangle=\langle\Gamma_3,\Gamma_4\rangle=\langle\Gamma_4,\Gamma_5\rangle=e,
\end{equation}
\begin{equation}
[\Gamma_1,\Gamma_5]=[\Gamma_2,\Gamma_4]=[\Gamma_2,\Gamma_5]=[\Gamma_3,\Gamma_5]=e,
\end{equation}
\begin{equation}
\Gamma_3=\Gamma_1\Gamma_2\Gamma_1.
\end{equation}
We can get the following relation too:
\begin{equation}
[\Gamma_4,\Gamma_1\Gamma_3\Gamma_1]=e.
\end{equation}
 This means that $\tilde{G}$ with the generators $\Gamma_1,\Gamma_2,\Gamma_3,\Gamma_4,\Gamma_5$ is isomorphic to $S_5$, so $\pi_1(X_{\text{Gal}})$  is trivial.
\end{proof}

\subsection{A union of the 4-point quartic degeneration and a plane}\label{extended-4}
In this subsection, we take a degeneration of a quartic surface to a plane arrangement with a $4$-point.

The Hilbert scheme of del Pezzo surfaces of degree 5 contains reducible surfaces that consist in a general degree 4 complete intersection $F$ of type $(2,2)$ in $\mathbb{CP}^ 4$ (which is itself a del Pezzo surface), plus a plane $P$ meeting $F$ along a line.

We can show that the configuration shown in Figure \ref{case6} is a limit of smooth del Pezzo surfaces of degree 5.
Indeed, the surface $F$ can degenerate to the union of 4 planes filling up the subdivided square in Figure \ref{case6}:
Simply degenerate the two quadrics cutting out $F$ in two general quadric cones with the same vertex (which is then the 4-tuple point $5$ that is common to the 4 planes). In this degeneration the number of lines contained in $F$, which is 16, is mapped to the following configuration of lines: Take a general quadric $Q$ in $\mathbb{CP}^4$, it cuts each of the 4 lines through 4-tuple point $5$ in two points; then on each of the 4 planes take the 4 lines pairwise joining the pairs of points not on the same line through 4-tuple point $5$, these are the 16 limits in question. Then a general plane through a line on $F$ goes to a plane like the ``rightmost top'' plane in Figure \ref{case6}.

\begin{figure}[H]
\begin{center}
\scalebox{0.60}{\includegraphics{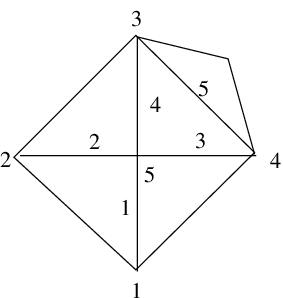}}
\end{center}
\setlength{\abovecaptionskip}{-0.15cm}
\caption{A union of the 4-point quartic degeneration and a plane}\label{case6}
\end{figure}

\begin{thm}\label{thm-7}
The fundamental group of the Galois cover of the surface with the degeneration as in Figure \ref{case6} is trivial.
\end{thm}

\begin{proof}
The branch curve $S$ in $\mathbb{CP}^2$ is an arrangement of 10 lines. We regenerate each vertex in turn and compute the group $G$.

Vertices 1 and 2 are 1-points, giving the braids  $Z_{1\; 1^{'}}$ and $Z_{2\; 2^{'}}$ respectively, and hence the following relations in $G$:
\begin{equation}\label{7-1}
\Gamma_1=\Gamma_1',~~~~~~\Gamma_2=\Gamma_2'.
\end{equation}
Vertex 3  is a 2-point and it gives the braid monodromy factors
 $Z^3_{4\;4^{'},5}$ and $(Z_{5\;5^{'}})^{Z^2_{4\;4^{'},5}}$. The relations in $G$ are:
\begin{equation}\label{7-2}
\langle\Gamma_4,\Gamma_5\rangle=\langle\Gamma_4',\Gamma_5\rangle=\langle\Gamma_4^{-1}\Gamma_4'\Gamma_4,\Gamma_5\rangle=e,
\end{equation}
\begin{equation}\label{7-3}
\Gamma_5'=\Gamma_5\Gamma_4'\Gamma_4\Gamma_5\Gamma_4^{-1}{\Gamma_4'}^{-1}\Gamma_5^{-1}.
\end{equation}
Vertex 4 is also a 2-point and it gives the braid monodromy factors
 $Z^3_{3\;3^{'},5}$ and $(Z_{5\;5^{'}})^{Z^2_{3\;3^{'},5}}$. The relations in $G$ are:
\begin{equation}\label{7-4}
\langle\Gamma_3,\Gamma_5\rangle=\langle\Gamma_3',\Gamma_5\rangle=\langle\Gamma_3^{-1}\Gamma_3'\Gamma_3,\Gamma_5\rangle=e,
\end{equation}
\begin{equation}\label{7-5}
\Gamma_5'=\Gamma_5\Gamma_3'\Gamma_3\Gamma_5\Gamma_3^{-1}{\Gamma_3'}^{-1}\Gamma_5^{-1}.
\end{equation}
The braid monodromy factors corresponding to the 4-point (vertex 5) were computed in \cite{Ogata}.
These braids give rise to the following relations in $G$:
\begin{equation}\label{7-6}
\langle\Gamma_1',\Gamma_2\rangle=\langle\Gamma_1',\Gamma_2'\rangle=\langle\Gamma_1',\Gamma_2^{-1}\Gamma_2'\Gamma_2\rangle=e,
\end{equation}
\begin{equation}\label{7-7}
\langle\Gamma_3,\Gamma_4\rangle=\langle\Gamma_3',\Gamma_4\rangle=\langle\Gamma_3^{-1}\Gamma_3'\Gamma_3,\Gamma_4\rangle=e,
\end{equation}
\begin{equation}\label{7-8}
[\Gamma_2'\Gamma_2\Gamma_1'\Gamma_2^{-1}{\Gamma_2'}^{-1},\Gamma_4] = e,
\end{equation}
\begin{equation}\label{7-9}
[\Gamma_2'\Gamma_2\Gamma_1'\Gamma_2^{-1}{\Gamma_2'}^{-1},\Gamma_3^{-1}{\Gamma_3'}^{-1}\Gamma_4^{-1}\Gamma_4'\Gamma_4\Gamma_3'\Gamma_3] = e,
\end{equation}
\begin{equation}\label{7-10}
\langle\Gamma_1,\Gamma_2\rangle=\langle\Gamma_1,\Gamma_2'\rangle=\langle\Gamma_1,\Gamma_2^{-1}\Gamma_2'\Gamma_2\rangle=e,
\end{equation}
\begin{equation}\label{7-11}
\langle\Gamma_3,\Gamma_4^{-1}\Gamma_4'\Gamma_4\rangle=\langle\Gamma_3',\Gamma_4^{-1}\Gamma_4'\Gamma_4\rangle=\langle\Gamma_3^{-1}\Gamma_3'\Gamma_3,\Gamma_4^{-1}\Gamma_4'\Gamma_4\rangle=e,
\end{equation}
\begin{equation}\label{7-12}
[\Gamma_2'\Gamma_2\Gamma_1\Gamma_2^{-1}{\Gamma_2'}^{-1},\Gamma_4^{-1}\Gamma_4'\Gamma_4] = e,
\end{equation}
\begin{equation}\label{7-13}
[\Gamma_2'\Gamma_2\Gamma_1\Gamma_2^{-1}{\Gamma_2'}^{-1}, \Gamma_3^{-1}{\Gamma_3'}^{-1}\Gamma_4^{-1}{\Gamma_4'}^{-1}\Gamma_4\Gamma_4'\Gamma_4\Gamma_3'\Gamma_3] = e,
\end{equation}
\begin{equation}\label{7-14}
\Gamma_2'\Gamma_2\Gamma_1'\Gamma_2\Gamma_1'^{-1}\Gamma_2^{-1}{\Gamma_2'}^{-1} = \Gamma_4\Gamma_3'\Gamma_4^{-1},
\end{equation}
\begin{equation}\label{7-15}
\Gamma_2'\Gamma_2\Gamma_1'\Gamma_2'\Gamma_1'^{-1}\Gamma_2^{-1}{\Gamma_2'}^{-1} = \Gamma_4\Gamma_3'\Gamma_3{\Gamma_3'}^{-1}\Gamma_4^{-1},
\end{equation}
\begin{equation}\label{7-16}
\Gamma_2'\Gamma_2\Gamma_1\Gamma_2\Gamma_1^{-1}\Gamma_2^{-1}{\Gamma_2'}^{-1} = \Gamma_4^{-1}\Gamma_4'\Gamma_4\Gamma_3'\Gamma_4^{-1}{\Gamma_4'}^{-1}\Gamma_4,
\end{equation}
\begin{equation}\label{7-17}
\Gamma_2'\Gamma_2\Gamma_1\Gamma_2'\Gamma_1^{-1}\Gamma_2^{-1}{\Gamma_2'}^{-1} = \Gamma_4^{-1}\Gamma_4'\Gamma_4\Gamma_3'\Gamma_3\Gamma_3'^{-1}\Gamma_4^{-1}{\Gamma_4'}^{-1}\Gamma_4.
\end{equation}
We also have the following parasitic and projective relations:
\begin{equation}\label{7-18}
[\Gamma_1, \Gamma_5]=[\Gamma_1', \Gamma_5]=[\Gamma_1, \Gamma_5']=[\Gamma_1', \Gamma_5']=e,
\end{equation}
\begin{equation}\label{7-19}
[\Gamma_2, \Gamma_5]=[\Gamma_2', \Gamma_5]=[\Gamma_2, \Gamma_5']=[\Gamma_2', \Gamma_5']=e,
\end{equation}
\begin{equation}\label{7-20}
\Gamma_5'\Gamma_5\Gamma_4'\Gamma_4\Gamma_3'\Gamma_3\Gamma_2'\Gamma_2\Gamma_1'\Gamma_1=e.
\end{equation}

By \eqref{7-14} and \eqref{7-15}, we have $\Gamma_3=\Gamma_3'$.

Combining it with \eqref{7-7}, \eqref{7-14}, and \eqref{7-15}, we get
\begin{equation}\label{7-21}
\Gamma_2'\Gamma_2\Gamma_1'\Gamma_2\Gamma_1'^{-1}\Gamma_2^{-1}{\Gamma_2'}^{-1} = \Gamma_3^{-1}\Gamma_4\Gamma_3.
\end{equation}

By \eqref{7-11} and \eqref{7-16} we have
\begin{equation}\label{7-22}
\Gamma_2'\Gamma_2\Gamma_1'\Gamma_2\Gamma_1'^{-1}\Gamma_2^{-1}{\Gamma_2'}^{-1} = \Gamma_3^{-1}\Gamma_4^{-1}\Gamma_4'\Gamma_4\Gamma_3.
\end{equation}
It follows that $\Gamma_4=\Gamma_4'$.

Thus, we get the following relations in $G$:
\begin{equation}\label{7-23}
\Gamma_1=\Gamma_1',~~~\Gamma_2=\Gamma_2',~~~\Gamma_3=\Gamma_3',~~~\Gamma_4
=\Gamma_4',~~~\Gamma_5'=\Gamma_4^{-2}\Gamma_5\Gamma_4^2,
\end{equation}
\begin{equation}\label{7-24}
\Gamma_2\Gamma_1\Gamma_2^{-1}=\Gamma_4\Gamma_3\Gamma_4^{-1},
\end{equation}
\begin{equation}\label{7-25}
\langle\Gamma_1,\Gamma_2\rangle=\langle\Gamma_3,\Gamma_4\rangle
=\langle\Gamma_3,\Gamma_5\rangle=\langle\Gamma_4,\Gamma_5\rangle=e,
\end{equation}
\begin{equation}\label{7-26}
[\Gamma_1,\Gamma_5]=[\Gamma_2,\Gamma_5]=
[\Gamma_2^2\Gamma_1\Gamma_2^{-2},\Gamma_4]=
[\Gamma_2^2\Gamma_1\Gamma_2^{-2},\Gamma_3^{-2}\Gamma_4\Gamma_3^2]=e,
\end{equation}
\begin{equation}
[\Gamma_4^2\Gamma_1\Gamma_4^{-2},\Gamma_5]=
[\Gamma_4^2\Gamma_2\Gamma_4^{-2},\Gamma_5]=e,
\end{equation}
\begin{equation}\label{7-27}
\Gamma_5\;\Gamma_4^2\;\Gamma_5\;\Gamma_3^2\;\Gamma_2^2\;\Gamma_1^2=e.
\end{equation}

In $\tilde{G}$, it is easy to see that the generators are $\Gamma_1,\Gamma_2,\Gamma_4,\Gamma_5$, and the relations are the following:
\begin{equation}\label{7-28}
\Gamma_1^2=\Gamma_2^2=\Gamma_4^2=\Gamma_5^2=e,
\end{equation}
\begin{equation}\label{7-29}
\langle\Gamma_1,\Gamma_2\rangle=\langle\Gamma_2,\Gamma_4\rangle=\langle\Gamma_4,\Gamma_5\rangle=e,
\end{equation}
\begin{equation}\label{7-30}
[\Gamma_1,\Gamma_4]=[\Gamma_1,\Gamma_5]=[\Gamma_2,\Gamma_5]=e.
\end{equation}

It is easy to see that $\tilde{G}\cong S_5$, thus $\pi_1(X_{\text{Gal}})$ is trivial.

\end{proof}

\subsection{A 5-point quintic degeneration}\label{central}
In this subsection we consider a quintic whose degeneration is depicted in Figure \ref{case5}. This degeneration gives a 5-point, in this case an intersection of five lines and also five planes.
According to \cite{CCFM}, the configuration in Figure \ref{case5} is a Zappatic surface of type $E_5$.
It is well-known that a general del Pezzo $S$ of degree $n$ in $\mathbb{CP}^ n$ can degenerate to a configuration of points of type $E_n, n=3,\ldots, 9$. Firstly, we degenerate $S$ \cite{CLM} to the cone over a general hyperplane section (elliptic curve) of $S$.
Secondly, we degenerate the hyperplane section to a cycle of lines.

\begin{figure}[htbp]
\begin{center}
\scalebox{0.68}{\includegraphics{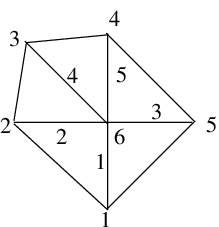}}
\end{center}
\setlength{\abovecaptionskip}{-0.3cm}
\caption{A 5-point quintic degeneration}\label{case5}

\end{figure}

The regeneration and the related braid monodromy of the 5-points were done in \cite{FT08}.  We use the result from \cite[Corollary\;2.5]{FT08} to give the braid monodromy relating to the 5-point.

\begin{thm}\label{thm-8}
The fundamental group of the Galois cover of the surface with a 5-point quintic degeneration as in Figure \ref{case5} is trivial.
\end{thm}

\begin{proof}
The branch curve $S$ in $\mathbb{CP}^2$ is an arrangement of 10 lines. We regenerate each vertex in turn and compute the group $G$.

Vertices 1, 2, 3, 4, and 5 are 1-points, therefore the related braid monodromy factors are  $Z_{1\;1^{'}}$, $Z_{2\;2^{'}}$, $Z_{4\;4^{'}}$, $Z_{5\;5^{'}}$, and $Z_{3\;3^{'}}$, respectively, and hence we have the following relations:
\begin{equation}\label{6-1}
\Gamma_1=\Gamma_1',~~~\Gamma_2=\Gamma_2',~~~\Gamma_3=\Gamma_3',~~~\Gamma_4=\Gamma_4',~~~\Gamma_5=\Gamma_5'.
\end{equation}
According to \cite[Corollary\;2.5]{FT08},
the braid monodromy corresponding to the 5-point yields the following relations in $G$:
\begin{equation}\label{6-2}
[\Gamma_3,\Gamma_4]=[\Gamma_3',\Gamma_4]=e,
\end{equation}
\begin{equation}\label{6-3}
\langle\Gamma_4',\Gamma_5\rangle=\langle\Gamma_4',\Gamma_5'\rangle=\langle\Gamma_4',\Gamma_5^{-1}\Gamma_5'\Gamma_5\rangle=e,
\end{equation}
\begin{equation}\label{6-4}
\langle\Gamma_2,\Gamma_4\rangle=\langle\Gamma_2',\Gamma_4\rangle=\langle\Gamma_2^{-1}\Gamma_2'\Gamma_2,\Gamma_4\rangle=e,
\end{equation}
\begin{equation}\label{6-5}
[\Gamma_4\Gamma_3\Gamma_4^{-1},\Gamma_5'\Gamma_5\Gamma_4'\Gamma_5^{-1}{\Gamma_5'}^{-1}]=
[\Gamma_4\Gamma_3'\Gamma_4^{-1},\Gamma_5'\Gamma_5\Gamma_4'\Gamma_5^{-1}{\Gamma_5'}^{-1}]=e,
\end{equation}
\begin{equation}\label{6-6}
\Gamma_4\Gamma_2'\Gamma_2\Gamma_4\Gamma_2^{-1}{\Gamma_2'}^{-1}\Gamma_4^{-1}=\Gamma_5'\Gamma_5\Gamma_4'\Gamma_5^{-1}{\Gamma_5'}^{-1},
\end{equation}
\begin{equation}\label{6-7}
[\Gamma_1,\Gamma_4]=[\Gamma_1',\Gamma_4]=
[\Gamma_1,\Gamma_5'\Gamma_5\Gamma_4'\Gamma_5^{-1}{\Gamma_5'}^{-1}]=[\Gamma_1',\Gamma_5'\Gamma_5\Gamma_4'\Gamma_5^{-1}{\Gamma_5'}^{-1}]=e,
\end{equation}
\begin{equation}\label{6-8}
\langle\Gamma_1',\Gamma_2\rangle=\langle\Gamma_1',\Gamma_2'\rangle=\langle\Gamma_1',\Gamma_2^{-1}\Gamma_2'\Gamma_2\rangle=e,
\end{equation}
\begin{equation}\label{6-9}
\langle\Gamma_4\Gamma_3\Gamma_4^{-1},\Gamma_5\rangle=\langle\Gamma_4\Gamma_3'\Gamma_4^{-1},\Gamma_5\rangle=\langle\Gamma_4\Gamma_3^{-1}\Gamma_3'\Gamma_3\Gamma_4^{-1},\Gamma_5\rangle=e,
\end{equation}
\begin{equation}\label{6-10}
\Gamma_2'\Gamma_2\Gamma_1'\Gamma_2\Gamma_1'^{-1}\Gamma_2^{-1}{\Gamma_2'}^{-1}=\Gamma_4^{-1}\Gamma_5\Gamma_4\Gamma_3'
\Gamma_4^{-1}\Gamma_5^{-1}\Gamma_4,
\end{equation}
\begin{equation}\label{6-11}
\Gamma_2'\Gamma_2\Gamma_1'\Gamma_2'\Gamma_1'^{-1}\Gamma_2^{-1}{\Gamma_2'}^{-1}=
\Gamma_4^{-1}\Gamma_5\Gamma_4\Gamma_3'\Gamma_3\Gamma_3'^{-1}\Gamma_4^{-1}\Gamma_5^{-1}\Gamma_4,
\end{equation}
\begin{equation}\label{6-12}
[\Gamma_2'\Gamma_2\Gamma_1'\Gamma_2^{-1}{\Gamma_2'}^{-1},\Gamma_4^{-1}\Gamma_5\Gamma_4]=e,
\end{equation}
\begin{equation}\label{6-13}
[\Gamma_3'\Gamma_3\Gamma_2'\Gamma_2\Gamma_1'\Gamma_2^{-1}{\Gamma_2'}^{-1}\Gamma_3^{-1}\Gamma_3'^{-1},
\Gamma_4^{-1}\Gamma_5^{-1}\Gamma_5'\Gamma_5\Gamma_4]=e,
\end{equation}
\begin{equation}\label{6-14}
\langle\Gamma_1,\Gamma_2\rangle=\langle\Gamma_1,\Gamma_2'\rangle=\langle\Gamma_1,\Gamma_2^{-1}\Gamma_2'\Gamma_2\rangle=e,
\end{equation}
\begin{equation} \label{6-15}
\begin{split}
\langle\Gamma_4\Gamma_3\Gamma_4^{-1},\Gamma_5^{-1}\Gamma_5'\Gamma_5\rangle&=\langle\Gamma_4\Gamma_3'\Gamma_4^{-1},
\Gamma_5^{-1}\Gamma_5'\Gamma_5\rangle\\
&=\langle\Gamma_4\Gamma_3^{-1}\Gamma_3'\Gamma_3\Gamma_4^{-1},\Gamma_5^{-1}\Gamma_5'\Gamma_5\rangle=e,
\end{split}
\end{equation}
\begin{equation}\label{6-16}
\Gamma_2'\Gamma_2\Gamma_1\Gamma_2\Gamma_1^{-1}\Gamma_2^{-1}{\Gamma_2'}^{-1}=
\Gamma_4^{-1}\Gamma_5^{-1}\Gamma_5'\Gamma_5\Gamma_4\Gamma_3'\Gamma_4^{-1}\Gamma_5^{-1}{\Gamma_5'}^{-1}\Gamma_5\Gamma_4,
\end{equation}
\begin{equation}\label{6-17}
\Gamma_2'\Gamma_2\Gamma_1\Gamma_2'\Gamma_1^{-1}\Gamma_2^{-1}{\Gamma_2'}^{-1}=
\Gamma_4^{-1}\Gamma_5^{-1}\Gamma_5'\Gamma_5\Gamma_4\Gamma_3'\Gamma_3\Gamma_3'^{-1}\Gamma_4^{-1}\Gamma_5^{-1}{\Gamma_5'}^{-1}\Gamma_5\Gamma_4,
\end{equation}
\begin{equation}\label{6-18}
[\Gamma_2'\Gamma_2\Gamma_1\Gamma_2^{-1}{\Gamma_2'}^{-1},\Gamma_4^{-1}\Gamma_5^{-1}\Gamma_5'\Gamma_5\Gamma_4]=e,
\end{equation}
\begin{equation}\label{6-19}
[\Gamma_3'\Gamma_3\Gamma_2'\Gamma_2\Gamma_1\Gamma_2^{-1}{\Gamma_2'}^{-1}\Gamma_3^{-1}\Gamma_3'^{-1},\Gamma_4^{-1}\Gamma_5^{-1}
{\Gamma_5'}^{-1}\Gamma_5\Gamma_5'\Gamma_5\Gamma_4]=e.
\end{equation}
We also have the following projective relation:
\begin{equation}\label{6-20}
\Gamma_5'\Gamma_5\Gamma_4'\Gamma_4\Gamma_3'\Gamma_3\Gamma_2'\Gamma_2\Gamma_1'\Gamma_1=e.
\end{equation}

The generators of  $\tilde{G}$ are $\Gamma_1,\Gamma_2,\Gamma_3,\Gamma_4,\Gamma_5$.
The relations are the following:
\begin{equation}
\Gamma_1^2=\Gamma_2^2=\Gamma_3^2=\Gamma_4^2=\Gamma_5^2=e,
\end{equation}
\begin{equation}\label{6-22}
\langle\Gamma_1,\Gamma_2\rangle=\langle\Gamma_2,\Gamma_4\rangle=
\langle\Gamma_3,\Gamma_5\rangle=\langle\Gamma_4,\Gamma_5\rangle=e,
\end{equation}
\begin{equation}\label{6-21}
[\Gamma_1,\Gamma_4]=[\Gamma_1,\Gamma_5]=[\Gamma_3,\Gamma_4]=e,
\end{equation}
\begin{equation}\label{6-23}
\Gamma_1\Gamma_2\Gamma_1=\Gamma_3\Gamma_4\Gamma_5\Gamma_4\Gamma_3.
\end{equation}
We eliminate generator $\Gamma_2$ from the list of generators, so now $\tilde{G}$ is generated by
$\Gamma_1, \Gamma_3, \Gamma_4, \Gamma_5$ and admits these relations:
\begin{equation}
\Gamma_1^2=\Gamma_3^2=\Gamma_4^2=\Gamma_5^2=e,
\end{equation}
\begin{equation}
\langle\Gamma_1,\Gamma_3\rangle=\langle\Gamma_3,\Gamma_5\rangle=\langle\Gamma_4,\Gamma_5\rangle=e,
\end{equation}
\begin{equation}\label{6-21}
[\Gamma_1,\Gamma_4]=[\Gamma_1,\Gamma_5]=[\Gamma_3,\Gamma_4]=e.
\end{equation}

Thus $\tilde{G}\cong S_5$, and it follows that $\pi_1(X_{\text{Gal}})$ is trivial.

\end{proof}

\subsection{A 4-point quintic degeneration}\label{sec7}
In this subsection we consider a quintic that degenerates to a union of planes as shown in Figure \ref{case8}. This degeneration gives a special 4-point, in this case an intersection of four edges.

A rational normal scroll $F$ of degree 5 in $\mathbb{CP}^ 6$ can degenerate to the cone over a hyperplane section of it, which is a rational normal curve $C$ of degree 5 in $\mathbb{CP}^ 5$. Then $C$ can be degenerated to a chain of lines which is well-known. This yields $F$ degenerates to 5 planes as in Figure \ref{case8}.

\begin{figure}[H]
\begin{center}
\scalebox{0.60}{\includegraphics{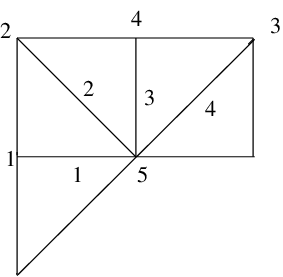}}
\end{center}
\setlength{\abovecaptionskip}{-0.15cm}
\caption{A 4-point quintic degeneration}\label{case8}
\end{figure}

\begin{thm}\label{thm-4}
The fundamental group of the Galois cover of the surface with a 4-point quintic degeneration as in Figure \ref{case8} is trivial.
\end{thm}

\begin{proof}
The branch curve $S$ in $\mathbb{CP}^2$ is an arrangement of 8 lines. We regenerate each vertex in turn and compute the group $G$.

Vertices 1, 2, 3, and 4 are 1-points, which give rise to the braids  $Z_{1 \; 1^{'}}$, $Z_{2 \; 2^{'}}$, $Z_{4\; 4^{'}}$, and $Z_{3 \; 3^{'}}$, respectively, and hence to the following relations in $G$:
\begin{equation}\label{9-1}
\Gamma_1=\Gamma_1',~~~\Gamma_2=\Gamma_2',~~~\Gamma_3=\Gamma_3',~~~\Gamma_4=\Gamma_4'.
\end{equation}
Vertex 5 is an outer $4$-point and it gives rise to the following relations:
	\begin{equation}\label{9-2}
		\langle{\Gamma_3},{\Gamma_4} \rangle = \langle {\Gamma_3'},{\Gamma_4} \rangle= \langle \Gamma_{3}^{-1}\Gamma_{3'}\Gamma_{3}, \Gamma_{4} \rangle =e,
	\end{equation}
	\begin{equation}\label{9-3}
	[\Gamma_{2'}\Gamma_{2}\Gamma_{1}\Gamma_{2}^{-1}\Gamma_{2'}^{-1}, \Gamma_{4}] = [\Gamma_{2'}\Gamma_{2}\Gamma_{1'}\Gamma_{2}^{-1}\Gamma_{2'}^{-1}, \Gamma_{4} ] = e,
	\end{equation}
	\begin{equation}\label{9-4}
		[\Gamma_{2},\Gamma_{4}] = [\Gamma_{2'},\Gamma_{4}] =e,
	\end{equation}
	\begin{equation}\label{9-5} [\Gamma_{3'}\Gamma_{3}\Gamma_{2'}\Gamma_{2}\Gamma_{1}\Gamma_{2}^{-1}\Gamma_{2}'^{-1}\Gamma_{3}^{-1}\Gamma_{3'}^{-1}, \Gamma_{4}^{-1}\Gamma_{4'}\Gamma_{4} ] = 	 [\Gamma_{3'}\Gamma_{3}\Gamma_{2'}\Gamma_{2}\Gamma_{1'}\Gamma_{2}^{-1}\Gamma_{2}'^{-1}\Gamma_{3}^{-1}\Gamma_{3'}^{-1}, \Gamma_{4}^{-1}\Gamma_{4'}\Gamma_{4} ] = e,
	\end{equation}
	\begin{equation}\label{9-6}
	[\Gamma_{3'}\Gamma_{3}\Gamma_{2}\Gamma_{3}^{-1}\Gamma_{3'}^{-1}, \Gamma_{4}^{-1}\Gamma_{4'}\Gamma_{4}] = 	 [\Gamma_{3'}\Gamma_{3}\Gamma_{2'}\Gamma_{3}^{-1}\Gamma_{3'}^{-1}, \Gamma_{4}^{-1}\Gamma_{4'}\Gamma_{4}] = e,
	\end{equation}
	\begin{equation}\label{9-7}
	\Gamma_{4}\Gamma_{3'}\Gamma_{3}\Gamma_{4}\Gamma_{3}^{-1}\Gamma_{3'}^{-1}\Gamma_{4}^{-1} = \Gamma_{4'},
	\end{equation}
	\begin{equation}\label{9-8}
\langle{\Gamma_1'},{\Gamma_2} \rangle = \langle {\Gamma_1'},{\Gamma_2'} \rangle= \langle \Gamma_{1'}, \Gamma_{2}^{-1}\Gamma_{2'}\Gamma_{2} \rangle =e,
	\end{equation}
	\begin{equation}\label{9-9}
	\Gamma_{1} = \Gamma_{2'}\Gamma_{2}\Gamma_{1'}\Gamma_{2}^{-1}\Gamma_{2'}^{-1},
	\end{equation}
	\begin{equation}\label{9-9}
	\begin{split}
	\langle \Gamma_{2'}\Gamma_{2}\Gamma_{1'} \Gamma_{2} \Gamma_{1'}^{-1}\Gamma_{2}^{-1}\Gamma_{2'}^{-1}, \Gamma_{4}\Gamma_{3}\Gamma_{4}^{-1} \rangle = \langle \Gamma_{2'}\Gamma_{2}\Gamma_{1'} \Gamma_{2'} \Gamma_{1'}^{-1}\Gamma_{2}^{-1}\Gamma_{2'}^{-1}, \Gamma_{4}\Gamma_{3}\Gamma_{4}^{-1} \rangle = \\
	=  \langle \Gamma_{2'}\Gamma_{2}\Gamma_{1'} \Gamma_{2}^{-1}\Gamma_{2'}\Gamma_{2} \Gamma_{1'}^{-1}\Gamma_{2}^{-1}\Gamma_{2'}^{-1}, \Gamma_{4}\Gamma_{3}\Gamma_{4}^{-1} \rangle = e,
	\end{split}
	\end{equation}
	\begin{equation}\label{9-10a} \Gamma_{2'}\Gamma_{2}\Gamma_{1'}\Gamma_{2}^{-1}\Gamma_{2'}^{-1}\Gamma_{1'}^{-1}\Gamma_{2}^{-1}\Gamma_{2'}^{-1}\Gamma_{4}\Gamma_{3}^{-1}
\Gamma_{3'}\Gamma_{3}\Gamma_{4}^{-1}\Gamma_{2'}\Gamma_{2}\Gamma_{1'}\Gamma_{2'}\Gamma_{2}\Gamma_{1'}^{-1}\Gamma_{2}^{-1}\Gamma_{2'}^{-1} = \Gamma_{4}\Gamma_{3}\Gamma_{4}^{-1},
	\end{equation}
	\begin{equation}\label{9-10b}
	[\Gamma_{1}, \Gamma_{4}\Gamma_{3}\Gamma_{4}^{-1}] = [\Gamma_{1'}, \Gamma_{4}\Gamma_{3}\Gamma_{4}^{-1}] = [\Gamma_{1}, \Gamma_{4}\Gamma_{3'}\Gamma_{4}^{-1}] =
[\Gamma_{1'}, \Gamma_{4}\Gamma_{3'}\Gamma_{4}^{-1}] = e.
	\end{equation}
We also have the following projective relation:
\begin{equation}\label{9-11}
\Gamma_4'\Gamma_4\Gamma_3'\Gamma_3\Gamma_2'\Gamma_2\Gamma_1'\Gamma_1=e.
\end{equation}

Using \eqref{9-1}, $\tilde{G}$ admits the following relations:
\begin{equation}\label{9-19}
\langle\Gamma_1,\Gamma_2\rangle=\langle\Gamma_2,\Gamma_3\rangle=\langle\Gamma_3,\Gamma_4\rangle=e,
\end{equation}
\begin{equation}\label{9-18}
[\Gamma_1,\Gamma_3] = [\Gamma_1,\Gamma_4] =[\Gamma_2,\Gamma_4] =e.
\end{equation}
 It follows that $\tilde{G}\cong S_5$. Thus $\pi_1(X_{\text{Gal}})$  is trivial.

\end{proof}

\end{document}